\documentclass{amsart}
\usepackage{amssymb}
\usepackage{amscd}
\usepackage[latin1]{inputenc}
\usepackage{amsmath}
\usepackage{amsfonts}
\usepackage{latexsym}
\usepackage{verbatim}
\usepackage[pdftex]{graphicx}
\usepackage{epsfig}
\usepackage{color}
\usepackage{caption}
\newtheorem{theorem}{Theorem}[section]
\newtheorem{corollary}[theorem]{Corollary}
\newtheorem{lemma}[theorem]{Lemma}

\newtheorem{proposition}[theorem]{Proposition}

\theoremstyle{definition}
\newtheorem{definition}[theorem]{Definition}
\newtheorem{example}[theorem]{Example}
\newtheorem{remark}[theorem]{Remark}

\numberwithin{equation}{section}

\begin{document}
\title[Follower, Predecessor, and Extender Set Sequences of $\beta$-Shifts]{Follower, Predecessor, and Extender Set Sequences of $\beta$-Shifts}

\begin{abstract}

Given a one-dimensional shift $X$, let $|F_X(n)|$ be the number of follower sets of words of length $n$ in $X$, and $|P_X(n)|$ be the number of predecessor sets of words of length $n$ in $X$. We call the sequence $\{|F_X(n)|\}_{n \in \mathbb{N}}$ the follower set sequence of the shift $X$, and $\{|P_X(n)|\}_{n \in \mathbb{N}}$ the predecessor set sequence of the shift $X$. Extender sets are a generalization of follower sets (see ~\cite{KassMadden}), and we define the extender set sequence similarly. In this paper, we examine achievable differences in limiting behavior of follower, predecessor, and extender set sequences. This is done through the classical $\beta$-shifts, first introduced in \cite{Renyi}. We show that the follower set sequences of $\beta$-shifts must grow at most linearly in $n$, while the predecessor and extender set sequences may demonstrate exponential growth rate in $n$, depending on choice of $\beta$.
\end{abstract}

\date{}
\author{Thomas French}
\address{Thomas French\\
Department of Mathematics\\
University of Denver\\
2280 S. Vine St.\\
Denver, CO 80208}
\email{Thomas.French@du.edu}

\maketitle

\section{Introduction}
\label{intro}

For any $\mathbb{N}$ or $\mathbb{Z}$ shift $X$ and finite word $w$ appearing in some point of $X$, the \textbf{follower set} of $w$, written $F_X(w)$, is defined as the set of all words $u$ such that the word $wu$ occurs in some point of $X$. The \textbf{predecessor set} of $w$, written $P_X(w)$, is the set of all words $s$ such that the word $sw$ occurs in some point of $X$. The \textbf{extender set} of $w$, written $E_X(w)$, is the set of all pairs of words $(s,u)$ such that the word $swu$ occurs in some point of $X$. It is well-known that for a $\mathbb{Z}$ shift $X$, finiteness of $\{F_X(w) \ | \ w \text{ in the language of $X$} \}$ is equivalent to $X$ being sofic, that is, the image of a shift of finite type under a continuous shift-commuting map. (see ~\cite{LindMarcus}) (In fact it is true that finiteness of $\{E_X(w) \ | \ w \text{ in the language of $X$}\}$ is equivalent to $X$ being sofic as well, see ~\cite{OrmesPavlov}).

We define the set $F_X(n)$ to be $\{F_X(w) \ | \ w \text{ has length }n\}$ for any positive integer $n$. Thus $|F_X(n)|$ is the total number of distinct follower sets which correspond to some word $w$ of length $n$ in $X$. $|P_X(n)|$ and $|E_X(n)|$ are defined similarly for predecessor and extender sets. Since the alphabet is finite, there are only finitely many words of a given length $n$, and so for any shift (sofic or not), $|F_X(n)|, |P_X(n)|$ and $|E_X(n)|$ are finite for every $n$. If $X$ is sofic, $\{F_X(w) \ | \ w \text{ in the language of $X$} \}$ is finite, and thus the follower set sequence $\{|F_X(n)|\}_{n \in \mathbb{N}}$ must be bounded, and similarly for the predecessor set sequence $\{|P_X(n)|\}_{n \in \mathbb{N}}$ and extender set sequence $\{|E_x(n)|\}_{n \in \mathbb{N}}$. 

In \cite{French}, the author showed that these sequences can exhibit surprising behavior. In the sofic case, these sequences may be eventually periodic rather than eventually constant, and the period, as well as the gaps between values, can be arbitrarily prescribed. The author also showed the existence of a non-sofic shift whose follower and extender set sequences were not monotone increasing.

In all of these examples, the follower, predecessor, and extender set sequences behaved similarly, oscillating with the same periods, or in the non-sofic case, exhibiting the same growth rates. In this paper we establish that this is not always the case, demonstrating examples in which the follower set sequence displays different limiting behavior than the predecessor and extender set sequences. We do this by use of the classical $\beta$-shifts, introduced in ~\cite{Renyi}. We show that the follower set sequence must grow linearly, while the predecessor and extender set sequence may grow exponentially.

We also show a wide class of complexity sequences which may be realized as predecessor set sequences of shift spaces. To do this, we prove that for any right-infinite sequence $d$ with complexity sequence $\Phi_n(d)$, the sequence $\{\Phi_n(d)+1\}$ may be realized as the predecessor set sequence of a $\beta$-shift for some $\beta$.

\section{Definitions and preliminaries}
\label{defns}
Let $A$ denote a finite set, which we will refer to as our alphabet. 

\begin{definition}
A \textbf{word} over $A$ is a member of $A^n$ for some $n \in \mathbb{N}$. We denote the length of a word $w$ by $|w|$.
\end{definition}

\begin{definition} For any words $v \in A^n$ and $w \in A^m$, we define the \textbf{concatenation} $vw$ to be the pattern in $A^{n+m}$ whose first $n$ letters are the letters forming $v$ and whose next $m$ letters are the letters forming $w$.
\end{definition}

\begin{definition} A word $w$ is a \textbf{prefix} of a right-infinite sequence $z$ if the first $|w|$-many letters of $z$ are the letters forming $w$. We denote the n-letter prefix of a sequence $z$ by $(z)_n$.
\end{definition}

\begin{definition} 
The \textbf{language} of a $\mathbb{Z}$ shift $X$, denoted by $L(X)$, is the set of all words which appear in points of $X$. For any finite $n \in \mathbb{N}$, $L_n(X) := L(X) \cap A^n$, the set of words in the language of $X$ with length $n$. The \textbf{complexity sequence} of a $\mathbb{Z}$ shift $X$ is $\Phi_n(X) = |L_n(X)|$ for every $n \in \mathbb{N}$. That is, the complexity sequence is the sequence which records the number of words of length $n$ appearing in some point of $X$ for every length $n$.
\end{definition}

\begin{definition}
For any one-dimensional shift $X$ over the alphabet $A$, and any word $w$ in the language of $X$, we define the \textbf{follower set of $w$ in $X$}, $F_X(w)$, to be the set of all finite words $u \in L(X)$ such that the word $wu$ occurs in some point of $X$. The \textbf{predecessor set of $w$ in $X$}, $P_X(w)$, is defined to be the set of all finite words $s \in L(X)$ such that the word $sw$ occurs in some point of $X$. In some works, the follower and predecessor sets have been defined to be the set of all one-sided infinite sequences (in $A^{\mathbb{N}}$ or $A^{-\mathbb{N}}$ for followers and predecessors, respectively) which may follow/precede $w$. This definition is equivalent for followers, and in the case of a two-sided shift, for predecessors as well. For a one-sided shift, of course, no infinite sequence may precede a word $w$. The results of this paper will apply for either definition in any case where either definition makes sense.
\end{definition}

\begin{definition}
For any one-dimensional shift $X$ over the alphabet $A$, and any word $w$ in the language of $X$, we define the \textbf{extender set of w in $X$}, $E_X(w)$, to be the set of all pairs $(s, u)$ where $s, u \in L(X)$ and the word $swu$ occurs in some point of $X$. Again, a definition replacing finite words with infinite sequences is equivalent in the two-sided case.
\end{definition}

\begin{remark}
For any word $w \in L(X)$, define a projection function $f_w:E_X(w) \rightarrow F_X(w)$ by $f(s,u) = u$. Such a function sends the extender set of $w$ onto the follower set of $w$. Any two words $w, v$ with the same extender set would have the property then that $f_w(E_X(w))= f_v(E_X(v))$, that is, that $w$ and $v$ have the same follower set. Similarly, words which have the same extender set also have the same predecessor set.
\end{remark}

\begin{definition}
For any positive integer $n$, define the set $F_X(n) = \{F_X(w) \ | \ w \in L_{n}(X)\}$. Thus the cardinality $|F_X(n)|$ is the number of distinct follower sets of words of length $n$ in $X$. Similarly, define $P_X(n) = \{P_X(w) \ | \ w \in L_n(X)\}$ and $E_X(n) = \{E_X(w) \ | \ w \in L_n(X)\}$, so that $|P_X(n)|$ is the number of distinct predecessor sets of words of length $n$ in $X$ and $|E_X(n)|$ is the number of distinct extender sets of words of length $n$ in $X$.
\end{definition}

\begin{definition}
Given a shift $X$, the \textbf{follower set sequence of $X$} is the sequence $\{|F_X(n)|\}_{n \in \mathbb{N}}$. The \textbf{predecessor set sequence of $X$} is the sequence $\{|P_X(n)|\}_{n \in \mathbb{N}}$. The \textbf{extender set sequence of $X$} is the sequence $\{|E_X(n)|\}_{n \in \mathbb{N}}$.
\end{definition}

\begin{example}
Let $X$ be a full shift on the alphabet $A$. Then any word $w \in L(X)$ may be followed legally by any word in $A^{\mathbb{N}}$, and thus the follower set of any word is the same. Hence there is only one follower set in a full shift. Similarly, there is only one predecessor and one extender set in a full shift. Then $\{|F_X(n)|\}_{n \in \mathbb{N}} = \{|P_X(n)|\}_{n \in \mathbb{N}} = \{|E_X(n)|\}_{n \in \mathbb{N}} = \{1, 1, 1, ...\}$.
\end{example}

\begin{example}
The even shift is the one-dimensional sofic shift with alphabet $\{0,1\}$ defined by forbidding odd runs of zeros between ones. It is a simple exercise to show that the even shift has three follower sets, $F(0), F(1),$ and $F(10)$. The follower set sequence of the even shift is $\{|F_X(n)|\}_{n \in \mathbb{N}} = \{2, 3, 3, 3, ...\}$. It is easy to verify that for any word $w$ in the language of the even shift, the follower set of $w$ is identical to the follower set of one of these three words.
\end{example}

\begin{definition} 
A one-dimensional shift $X$ is a \textbf{shift of finite type} if there exists a finite collection of words $\mathcal{F}$ such that a point $x$ is in $X$ if no word from $\mathcal{F}$ ever appears in $x$.
\end{definition}

\begin{definition}
A one-dimensional shift $X$ is \textbf{sofic} if it is the image of a shift of finite type under a continuous shift-commuting map. 
\end{definition}

Equivalently, a shift $X$ is sofic iff there exists a finite directed graph $\mathcal{G}$ with labeled edges such that the points of $X$ are exactly the sets of labels of infinite (or in the two-sided case, bi-infinite) walks on $\mathcal{G}$. Then $\mathcal{G}$ is a \textbf{presentation} of $X$ and we say $X = X_\mathcal{G}$ (or that $X$ is the \textbf{edge shift} presented by $\mathcal{G}$). Another well-known equivalence is that sofic shifts are those with only finitely many follower sets, that is, a shift $X$ is sofic iff $\{F_X(w) \ | \ w \text{ in the language of $X$} \}$ is finite. The same equivalence exists for extender sets: X is sofic iff $\{E_X(w) \ | \ w \text{ in the language of $X$} \}$ is finite. (see ~\cite{OrmesPavlov}) This necessarily implies that for a sofic shift $X$, the follower set sequence and extender set sequence of $X$ are bounded. In fact, the converse is also true: if the follower set or extender set sequence of a shift $X$ is bounded, then $X$ is necessarily sofic. (See ~\cite{OrmesPavlov})

\begin{definition}
A word $v$ is \textbf{synchronizing} if for any words $u$ and $w$ such that $uv$ and $vw$ are both in the language of $X$, the word $uvw$ is also in the language of $X$. For a sofic shift, there exists a presentation $\mathcal{G}$ of $X$ such that every path labeled $v$ terminates at the same vertex.
\end{definition}

\begin{definition}
Suppose there is an order $<$ on the alphabet $\mathcal{A}$. Given two distinct words of the same length $x= x_1x_2...x_k$ and $y = y_1y_2...y_k$, we say that $x$ is \textbf{lexicographically less than} $y$ (denoted $x \prec y$) if there exists $1 \leq i \leq k$ such that for all $j <i$, $x_j = y_j$, but $x_i < y_i$. That is, at the first place where $x$ and $y$ disagree, $x$ takes a lesser value than $y$. We may also compare two right-infinite sequences using the lexicographic order.
\end{definition}

\begin{definition}
Two-sided shifts may be constructed from one-sided shifts using \textbf{the natural extension}: Given a one-sided shift $X$, create a two-sided shift $\widehat{X}$ by asserting that a bi-infinite sequence $x$ is in $\widehat{X}$ if and only if every finite subword of $x$ is in $L(X)$.
\end{definition}

\section{Follower, Predecessor, and Extender Set Sequences of $\beta$-shifts}
\label{Beta}

Given $\beta > 1$, consider the map $T_\beta: [0,1) \rightarrow [0,1)$ given by $T_\beta(x) = \beta x \pmod 1$. Let $d_\beta:[0,1) \rightarrow \{ \lfloor \beta \rfloor +1\}^\mathbb{N}$ be the map which sends each point $x \in [0,1)$ to its expansion in base $\beta$. That is, if $\displaystyle x = \sum_{n=1}^\infty \frac{x_n}{\beta^n}$, then $d_\beta(x) = .x_1x_2x_3...$. (In the case where $x$ has more than one $\beta$ expansion, we take the lexicographically largest expansion.) The closure of the image, $\overline{d_\beta([0,1))}$, is a one-sided symbolic dynamical system called the $\beta$-shift. (Introduced in \cite{Renyi}). It is easily established that the $\beta$-shift is measure-theoretically conjugate to the map $T_\beta$: If we partition $[0,1)$ into the intervals $[\frac{k}{\beta}, \frac{k+1}{\beta})$ for $k = 0, 1, ... , \lfloor \beta \rfloor -1$, along with $[\frac{\lfloor \beta \rfloor}{\beta} , 1)$, labeling each interval with its corresponding $k$, then for a given $x \in [0,1)$, $x_i$ will be the label of the partition element in which $T_\beta^i(x)$ resides. 

An equivalent characterization of the $\beta$-shift is given by a right-infinite sequence $d_\beta^* (1) = \displaystyle \lim_{x \nearrow 1} d_\beta(x)$. For any sequence $x$ on the alphabet $\{0, ..., \lfloor \beta \rfloor \}$, $x \in X_\beta$ if and only if every shift of $x$ is lexicographically less than or equal to $d_\beta^* (1)$ (see ~\cite{Parry}). Then clearly, the sequence $d_\beta^* (1)$ has the property that every shift of $d_\beta^* (1)$ is lexicographically less than or equal to $d_\beta^* (1)$. The sequence $d_\beta(1)$ only terminates with an infinite string of 0's in the case that $X_\beta$ is a shift of finite type, however, the sequence $d_\beta^*(1)$ never terminates with an infinite string of 0's. Indeed, $d_\beta^*(1) \neq d_\beta(1)$ if and only if $X_\beta$ is a shift of finite type (see ~\cite{Blanchard}).

\begin{example}\label{GMS}
The golden mean shift $X_\varphi$, the shift on alphabet $\{0,1\}$ in which the word 11 may not appear, considered as a one-sided shift, is the $\beta$-shift corresponding to $\beta = \varphi = \frac{1+\sqrt{5}}{2}$. Since the golden mean shift is a shift of finite type, we should have that $d_\varphi(1)$ terminates in an infinite string of 0's, and $d_\varphi^*(1) \neq d_\varphi(1)$. Indeed, $d_\varphi(1) = .11000000...$, and $d_\varphi^*(1) = .1010101010...$ The reader may check that the requirement that every shift of a right infinite sequence be lexicographically less than or equal to $d_\varphi^*(1)$ is equivalent to the requirement that the right-infinite sequence never see the word 11.
\end{example} 

In fact, any right-infinite sequence $d$ satisfying the two properties we have discussed must be equal to $d_\beta^*(1)$ for some $\beta$:

\begin{lemma}\label{twoprops}
Let $d$ be a one-sided right-infinite sequence on the alphabet $\mathcal{A} = \{0, 1, ..., k\}$ (with $k$ occurring in $d$) such that every shift of $d$ is lexicographically less than or equal to $d$ and $d$ does not end with an infinite string of zeros. Then there exists $k < \beta \leq k+1$ such that $d = d_\beta^*(1)$.
\end{lemma}
\begin{proof} Suppose $d = .d_0d_1d_2...$ be a sequence on the alphabet $\mathcal{A} = \{0, 1, ..., k\}$ with the symbol $k$ occuring in $d$. Let $\displaystyle 1 = \sum_{i=0}^\infty \frac{d_i}{\beta^{i+1}}$. Because not all $d_i$ are equal to zero, the equation has some solution $\beta$ by the Intermediate Value Theorem. Because $d$ does not end in an infinite string of 0's, it cannot be the case that $d$ is an expansion which is equal to $d_\beta(1) \neq d_\beta^*(1)$ as in the case where $X_\beta$ is a shift of finite type. Furthermore, the requirement that every shift of $d$ is lexicographically less than or equal to $d$ necessitates that $d_0 = k$, for if $d_0 < k$ but $d_i = k$ for some $i \neq 0$, then $\sigma^i(d) \succ d$, a contradiction. Then, multiplying both sides of the equation by $\beta$, we get 
$$\beta = \sum_{i=0}^\infty \frac{d_i}{\beta^{i}} = d_0 + \sum_{i=1}^\infty \frac{d_i}{\beta^i} = k + \sum_{i=1}^\infty \frac{d_i}{\beta^{i}}> k$$
with the final inequality being strict because not all of $\{d_i\}_{i=1}^\infty$ may be equal to 0. On the other hand,
$$\beta = \sum_{i=0}^\infty \frac{d_i}{\beta^{i}} \leq \sum_{i=0}^\infty \frac{k}{\beta^{i}} = k \sum_{i=0}^\infty \Big(\frac{1}{\beta}\Big)^i = k\Big(\frac{1}{1-\frac{1}{\beta}}\Big) = k \Big(\frac{\beta}{\beta -1}\Big)$$
as $\beta > 1$. But then
\begin{align*}
\beta &\leq k \Big(\frac{\beta}{\beta -1}\Big) \\
\beta(\beta - 1) &\leq k\beta \\
\beta - 1 &\leq k \\
\beta &\leq k + 1.
\end{align*}
\end{proof}

Moreover, $X_\beta$ is sofic if and only if $d_\beta^*(1)$ is eventually periodic (see \cite{Blanchard}). We use these facts to characterize the follower set sequences of all one-sided $\beta$-shifts:

\begin{lemma}\label{safewords}
If $w,v \in L(X_\beta)$ and $w$ does not terminate with any prefix of the sequence $d_\beta^* (1)$, then $wv \in L(X_\beta)$. Thus, $F_{X_\beta}(w) = X_\beta = F(\emptyset)$.
\end{lemma}

\begin{proof}
If $v \in L(X_\beta)$, then $v0^\infty \in X_\beta$ trivially. Then we claim $wv0^\infty \in x_\beta$ as well. Given $\sigma^i(wv0^\infty)$, if $i \geq |w|$, $\sigma^i(wv0^\infty) \preceq d_\beta^*(1)$ as $v0^\infty$ is legal. If $i < |w|$, then the first $|w| - i$ letters of $\sigma^i(wv0^\infty)$ are not a prefix of $d_\beta^*(1)$, but since they are contained in the legal word $w$, must be lexicographically less than the first $|w|-i$ letters of $d_\beta^*(1)$, and thus, $\sigma^i(wv0^\infty) \preceq d_\beta^*(1)$. Thus, $wv0^\infty \in X_\beta$ and $wv \in L(X_\beta)$.
\end{proof}

\begin{theorem}\label{FSConj}
For any $\beta$-shift, $|F_{X_\beta}(n)| \leq n$ for any $n$ if and only if $X_\beta$ is sofic.
\end{theorem}

\begin{proof}
Since $\beta > 1$, we have $\lfloor \beta \rfloor \geq 1$. Clearly $d_\beta(1-\epsilon)$ for very small $\epsilon > 0$ must begin with $\lfloor \beta \rfloor$, and so $d_\beta^*(1)$ must begin with $\lfloor \beta \rfloor$, which is lexicographically larger than $0$, and so by Lemma \ref{safewords}, $F(0) = F(\emptyset)$ for any $\beta$-shift. It was shown in \cite{FOP} that having a word $w \in L(X)$ with $F(w)=F(\emptyset)$ implies the conjecture of Ormes and Pavlov: $|F_X(n)| \leq n$ for some $n$ if and only if $X$ is sofic. Thus, this conjecture holds for every $\beta$-shift.
\end{proof}

\begin{theorem}\label{UBFSSbeta}
For any $\beta$-shift, and for any $n \in \mathbb{N}$, we have $|F_{X_\beta}(n)| \leq n + 1$.
\end{theorem}

\begin{proof}
Fix $n \in \mathbb{N}$ and partition the words $w \in L_n(X_\beta)$ into $n+1$ classes $\{S_0, S_1, ... S_n\}$ based on the maximal prefix of $d_\beta^* (1)$ appearing as a suffix of the word $w$. So the class $S_0$ will contain words which do not contain any prefix of $d_\beta^* (1)$ as a suffix, while the words in the class $S_1$ are words ending with the first letter of $d_\beta^* (1)$, but containing no larger prefix of $d_\beta^* (1)$ as a suffix, and so on, up to $S_n$, which consists of the single word formed by the first $n$ letters of $d_\beta^* (1)$. We claim that for two words $w$ and $v \in L_n(X_\beta)$ such that $w$ and $v$ are in the same class, $F(w) = F(v)$. This will imply the result.

By Lemma \ref{safewords}, any word $w \in S_0$ will have $F(w) = F(\emptyset)$, and so the claim is certainly true for $S_0$. $S_n$ contains only one word, so the claim holds trivially for $S_n$ as well. So let $0 < k < n$ and consider two words $v, w \in S_k$. Then the last $k$ letters of $v$ and $w$ are identical, and equal to the first $k$ letters of $d_\beta^* (1)$, and so $w$ and $v$ may only differ on their first $n -k$ letters. Let $s$ be some word in $F(w)$, so there exists some right-infinite sequence $z$ such that $wsz$ is a point of $X$. Then every shift of $wsz$ is lexicographically less than or equal to $d_\beta^* (1)$. Then, certainly, every shift of $vsz$ beyond the first $n-k$ shifts is lexicographically less than or equal to $d_\beta^* (1)$, because after the first $n-k$ letters, $vsz$ and $wsz$ are equal. Furthermore, for $0 \leq i < n-k$, $\sigma^i(v)$ is not a prefix of $d_\beta^* (1)$ by definition of the class $S_k$, and so $\sigma^i(v)$ is already strictly lexicographically less than the first $n - i$ letters of $d_\beta^* (1)$ because $v$ is a legal word. Then, any sequence $sz$ may follow $\sigma^i(v)$ and $\sigma^i(vsz)$ will be less lexicographically than $d_\beta^* (1)$. So in fact, the first $n-k$ letters place no restrictions whatsoever on the words which may follow $v$, and indeed, $s \in F(v)$. So $F(w) \subseteq F(v)$ and similarly, $F(v) \subseteq F(w)$, so $F(v) = F(w)$. Hence, for any $\beta$-shift $X_\beta$, $|F_{X_\beta}(n)| \leq n+1$ for all $n \in \mathbb{N}$.
\end{proof}

\begin{corollary}\label{FSSNonsoficBeta}
For any nonsofic $\beta$-shift $X_\beta$, $\{|F_{X_\beta}(n)|\}_{n \in \mathbb{N}} = \{n+1\}_{n \in \mathbb{N}}$.
\end{corollary}
\begin{proof}
By Thm. \ref{UBFSSbeta}, $|F_{X_\beta}(n)| \leq n + 1$ for all $n \in \mathbb{N}$. If $X_\beta$ is non-sofic, by Lemma \ref{FSConj},  $|F_{X_\beta}(n)| \geq n+1$ for all $n \in \mathbb{N}$. Thus, for all $n \in \mathbb{N}$ we have $|F_{X_\beta}(n)| = n+1$, that is, the follower set sequence $\{|F_{X_\beta}(n)|\}_{n \in \mathbb{N}}$ of $X_\beta$ is equal to $\{n+1\}_{n \in \mathbb{N}}$.
\end{proof}

\begin{remark} 
The shift $X_\beta$ contains a synchronizing word if and only if there exists some word $w \in L(X_\beta)$ such that $w$ is not a subword of the sequence $d_\beta^* (1)$ (see~\cite{Blanchard}). This provides examples of non-sofic shifts which contain a synchronizing word and still have the property that $|F_X(n)| = n+1$ for all $n \in \mathbb{N}$. Non-sofic $\beta$-shifts are furthermore remarkable because, unlike Sturmian shifts and other examples in which $|F_X(n)| = n+1$ for all $n$, non-sofic $\beta$-shifts have positive topological entropy. ($\beta > 1$ and the topological entropy of a $\beta$-shift is $\log(\beta)$).
\end{remark}

Though this is sufficient to characterize the follower set sequences of non-sofic beta-shifts, we may also prove that non-sofic beta-shifts have $n+1$ follower sets of words of length $n$ for every $n$ using the following lemma, which will be useful later:

\begin{lemma}\label{prefixes} Let $X_\beta$ be a $\beta$-shift. Let $w \in L_n(X_\beta)$ such that $w$ contains multiple different prefixes of $d_\beta^* (1)$ as suffixes (all necessarily nested subwords). Say the longest prefix of $d_\beta^* (1)$ contained as a suffix of $w$ has length $j$. Then $\sigma^j(d_\beta^* (1)) \preceq \sigma^k(d_\beta^* (1))$ for every $k$ such that $w$ has a prefix of $d_\beta^* (1)$ of length $k$ as a suffix, and therefore the follower set of $w$ consists of exactly all finite prefixes of the sequences which are lexicographically less than or equal to $\sigma^j(d_\beta^* (1))$.
\end{lemma}

\begin{proof}
Suppose $k < j$ and $w$ contains a prefix of $d_\beta^* (1)$ of length $j$ and a prefix of $d_\beta^* (1)$ of length $k$ as suffixes. We will denote the prefix of length $j$ by $(d_\beta^* (1))_j$, and similarly for $k$. Note that the first $k$ and the last $k$ letters of $(d_\beta^* (1))_j$ must be equal to $(d_\beta^* (1))_k$. So $d_\beta^* (1)$ might look like:\newline \newline

\includegraphics[scale=1]{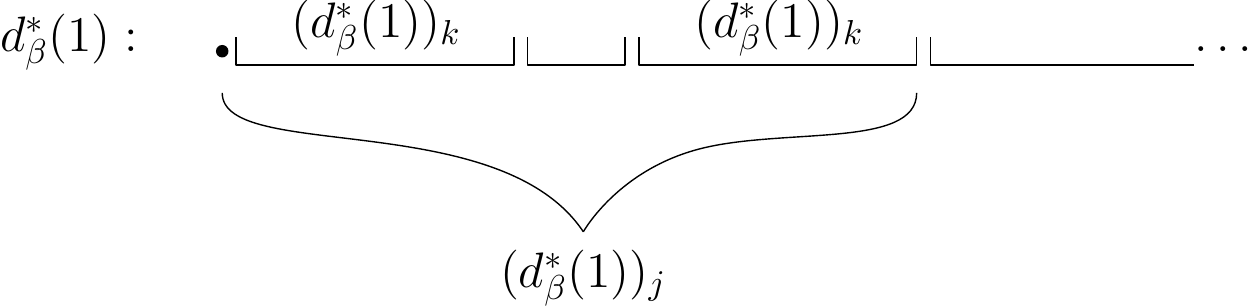} \newline

(It is also possible that no letters exist in the middle portion between the two occurences of $(d_\beta^* (1))_k$, or even that they overlap; the proof still holds in such cases). \newline
If $\sigma^j(d_\beta^* (1)) \succ \sigma^k(d_\beta^* (1))$, then it is easy to see that $\sigma^{j-k}(d_\beta^* (1)) \succ d_\beta^* (1)$, a contradiction, as every shift of $d_\beta^* (1)$ must be lexicographically less than or equal to $d_\beta^* (1)$. Therefore $\sigma^j(d_\beta^* (1)) \preceq \sigma^k(d_\beta^* (1))$, establishing the result. Since the right-infinite sequences which may legally follow $w$ consist of the sequences $z$ st. every shift of $wz$ is lexicographically less than or equal to $d_\beta^* (1)$, we have \newline
$\begin{aligned}
F(w) = \{s \> | & \> s \text{ is a prefix of some right-infinite sequence }z \\
&\text{ s.t. }z \preceq \sigma^k(d_\beta^* (1)) \> \forall \> k \text{ s.t. }w \text{ has a }k \text{-letter prefix of }d_\beta^* (1) \text{ as a suffix}\}.
\end{aligned}$ \newline
Since $\sigma^j(d_\beta^* (1))$ is the smallest such shift, the follower set of $w$ consists exactly of prefixes of those sequences lexicographically less than or equal to $\sigma^j(d_\beta^* (1))$.
\end{proof}

\begin{corollary}\label{betacorrollary} Let $X_\beta$ be a $\beta$-shift, $n \in \mathbb{N}$, and the classes $S_0, S_1, ..., S_n$ be as defined in the proof of Theorem \ref{UBFSSbeta}. Then the follower set of words in class $S_j$ is equal to the follower set of words in class $S_k$ if and only if $\sigma^j(d_\beta^* (1)) = \sigma^k(d_\beta^* (1))$.
\end{corollary}

\begin{proof} Lemma \ref{prefixes} shows that we may distinguish between follower sets of words by only looking at $\sigma^j(d_\beta^* (1))$, where $j$ is the length of the longest prefix of $d_\beta^* (1)$ appearing as a suffix of the word. The corollary follows immediately from that observation.
\end{proof}

It is easy to characterize the follower sets of a non-sofic $\beta$-shift from this corollary: since $d_\beta^* (1)$ is not eventually periodic, $j \neq k$ implies $\sigma^j(d_\beta^* (1)) \neq \sigma^k(d_\beta^* (1))$, and therefore the follower sets in each class $S_0, S_1, ... S_n$ must all be distinct. Thus, for each $n$, $|F_{X_\beta}(n)| = n+1$ (confirming a result we already established in Corollary \ref{FSSNonsoficBeta}). On the other hand, the corollary will also aid in characterizing the follower set sequences of sofic $\beta$-shifts: 

\begin{theorem}\label{sofic beta}
For any sofic $\beta$-shift, let $p = \min\{j \in \mathbb{N} ~ | ~ \exists ~ k<j, \ \sigma^k(d_\beta^* (1)) = \sigma^j(d_\beta^* (1)) \}$. Then for any $n \in \mathbb{N}$, we have: $|F_{X_\beta}(n)| =  \begin{cases} 
      n+1 & n < p \\
      p & n \geq p 
   \end{cases}
$
\end{theorem}
\begin{proof}
If $X_\beta$ is sofic sequence $d_\beta^* (1)$ is eventually periodic (see ~\cite{Blanchard}). Let $p = \min\{j \in \mathbb{N} ~ | ~ \exists ~ k<j, \ \sigma^k(d_\beta^* (1)) = \sigma^j(d_\beta^* (1)) \}$. Then for every length $n < p$, the shifts $\sigma^i(d_\beta^* (1))$, $0 \leq i \leq n$, are all distinct, and so $|F_{X_\beta}(n)| = n+1$ by Cor. \ref{betacorrollary}. However, at length $p$, the word in $S_p$ will have the same follower set as words in length $k$ for some $k < j$ by Cor. \ref{betacorrollary}, and so there will only be $p$ (rather than $p+1$) follower sets of words of length $p$. But if $\sigma^p(d_\beta^* (1)) = \sigma^k(d_\beta^* (1))$, then $\sigma^{p+\ell}(d_\beta^* (1)) = \sigma^{k+\ell}(d_\beta^* (1))$ for all $\ell \in \mathbb{N}$, and so after length $p$, no lengths will contribute any new follower sets. Indeed, $|F_{X_\beta}(n)| = p$ for all $n \geq p$.
\end{proof}

We next explore the predecessor set sequences of $\beta$-shifts. 

\begin{theorem}\label{PSSBeta}
For any $\beta$-shift $X_\beta$, the predecessor set sequence $\{|P_{X_\beta}(n)|\}_{n \in \mathbb{N}}$ of $X_\beta$ is equal to the complexity sequence $\{\Phi_n(d_\beta^*(1))\}_{n \in \mathbb{N}}$ of $d_\beta^*(1)$.
\end{theorem}
\begin{proof}
First, partition all finite words into classes $S_0, S_1, ...$ as before, indexed by the maximal prefix of $d_\beta^* (1)$ appearing as a suffix of the word. So $S_0$ is the set of words containing no prefix of $d_\beta^* (1)$ as a suffix, $S_1$ is the set of words terminating with the first letter of $d_\beta^* (1)$ but containing no larger prefix as a suffix, and so on. Note that, as we are considering all finite words, not just finite words of a fixed length, there will be infinitely many classes, each containing words of many different lengths.

Let $k, n \in \mathbb{N}$ and $w \in L_n(X_\beta)$. Then all of the words in $S_k$ are in the predecessor set of $w$ if and only if $w \preceq (\sigma^k(d_\beta^* (1)))_n$ (we are implicitly using Lemma \ref{prefixes} to rule out the possibility that a prefix of $d_\beta^* (1)$ of length less than $k$ could introduce some illegal word). For the rest of this proof, when a word $w$ satisfies the condition $w \preceq (\sigma^k(d_\beta^* (1)))_n$, we will simply say that w \textit{satisfies condition $k$}. Since $k$ may range from $0$ to infinity, while $n$ is fixed, every $n$-letter word in $d_\beta^* (1)$ will appear as the upper bound in condition $k$ for some $k$. Moreover, these predecessor sets are nested--if $S_k$ is part of the predecessor set of a word $w$, then it is also part of the predecessor set of every lexicographically smaller word of the same length. So, given a word $w$, the predecessor set of $w$ depends only on how many of $\Phi_{d_\beta^*(1)}(n)$ different conditions on $w$ are met. But since these conditions are nested, there are only $\Phi_{d_\beta^* (1)}(n)+1$ (nested) subsets of these conditions which may be simultaneously met. Moreover, $w$ must satisfy at least \textit{one} of them, because for any legal word $w \in L_n(X_\beta)$, $w \preceq (d_\beta^* (1))_n$, so $w$ satisfies condition 0. If this is the only one of the conditions satisfied, then the predecessor set of $w$ only consists of words in $S_0$ and $S_i$ for any $i$ such that $(\sigma^i(d_\beta^*(1)))_n = (d_\beta^*(1))_n$ (that is, that condition $i$ and condition 0 are the same condition). If $w$ satisfies this condition and only one other, the predecessor set of $w$ contains $S_0$ and any such $S_i$ as before, along with any $S_k$ such that $(\sigma^k(d_\beta^* (1)))_n$ is equal to the lexicographically greatest $n$-letter word in $d_\beta^* (1)$ besides $(d_\beta^* (1))_n$, and so on. In other words, words in $S_k$ are part of the predecessor set of $w$ when $w$ satisfies condition $k$, but there are only $\Phi_{d_\beta^* (1)}(n)$-many subsets of $\mathbb{N}$ which are realizable as the set of all $k$ for which $w$ satisfies condition $k$. Thus, for all $n$, there are $\Phi_{d_\beta^* (1)}(n)$ predecessor sets of length $n$ in $X_\beta$. So the predecessor set sequence of any $\beta$-shift is $\{\Phi_{d_\beta^* (1)}(n)\}_{n \in \mathbb{N}}$.
\end{proof}

This is enough to see that the predecessor and follower set sequences of $\beta$-shifts may exhibit vastly different limiting behavior:

\begin{example}
Let $d$ be any right-infinite sequence on $\{0,1\}$ which contains every word in $L(X_{[2]})$ as a subword, that is, that $\mathcal{O}(d) = \{\sigma^i(d) \> | \> i \in \mathbb{N}\}$ is dense in $\{0,1\}^\mathbb{N}$. Let $\tilde{d} = .2d$, so the first digit of $\tilde{d}$ is 2 and $\sigma(\tilde{d}) = d$. Then $\tilde{d}$ satisfies the hypothesis of Lemma \ref{twoprops} and so there exists some $2 < \beta \leq 3$ such that $\tilde{d} = d_\beta^*(1)$. Since $\tilde{d}$ is certainly not eventually periodic, $X_\beta$ is non-sofic, and so the follower set sequence of $X_\beta$ is $\{|F_{X_\beta}(n)|\} = \{n+1\}$ by Corollary \ref{FSSNonsoficBeta}. On the other hand, by Theorem \ref{PSSBeta}, the predecessor set sequence of $X_\beta$ is $\{|P_{X_\beta}(n)|\} = \{\Phi_{\tilde{d}}(n)\} = \{2^n + 1\}$. This shows that the predecessor set sequence may grow exponentially in $n$ even when the follower set sequence grows only linearly in $n$.
\end{example}

We see that the complexity sequence of any sequence which serves as $d_\beta^* (1)$ for some $\beta$ may be realized as the predecessor set sequence of a shift space. This leads us to a new question: May any complexity sequence appear as the follower or predecessor set sequence of some shift? The following propositions answer the question in the affirmative for a wide class of sequences:

\begin{proposition}\label{achievable}
Let $\{\Phi_d(n)\}$ be the complexity sequence of a right-infinite sequence $d$ such that for all $i \in \mathbb{N}$, $\sigma^i(d) \preceq d$. Then the sequence $\{\Phi_d(n)\}$ is the predecessor set sequence of $X_\beta$ for some $\beta > 1$.
\end{proposition}

\begin{proof} 
If $d$ is such that for all $i \in \mathbb{N}$, $\sigma^i(d) \preceq d$, and $d$ does not end in an infinite string of 0's, then by Lemma \ref{twoprops} there exists $\beta$ such that $d = d_\beta^*(1)$, and by Theorem \ref{PSSBeta}, $\{|P_{X_\beta}(n)|\} = \{\Phi_d(n)\}$. If $d$ is such that for all $i \in \mathbb{N}$, $\sigma^i(d) \preceq d$, but $d$ does end in an infinite string of 0's, define a new sequence $\hat{d}$ such that each digit of $\hat{d}$ is one greater than the corresponding digit of $d$. Then $\hat{d}$ has the same complexity sequence as $d$, and also satisfies the property that for all $i \in \mathbb{N}$, $\sigma^i(\hat{d}) \preceq \hat{d}$, and so by Lemma \ref{twoprops} there exists $\beta$ such that $\hat{d} = d_\beta^*(1)$, and by Theorem \ref{PSSBeta}, $\{|P_{X_\beta}(n)|\} = \{\Phi_{\hat{d}}(n)\} = \{\Phi_d(n)\}$.
\end{proof}

\begin{proposition}\label{PlusOne}
Let $\{\Phi_d(n)\}$ be the complexity sequence of any right-infinite sequence $d$. Then the sequence $\{\Phi_d(n)+1\}$ is the predecessor set sequence of $X_\beta$ for some $\beta > 1$.
\end{proposition}

\begin{proof}
Suppose $d$ does not satisfy the property that for all $i \in \mathbb{N}$, $\sigma^i(d) \preceq d$. (We may assume that we have already added 1 to each digit, if necessary, so that $d$ does not end in a string of 0's). Then create a new sequence $\tilde{d}$ so that $\sigma(\tilde{d}) = d$, but the first letter of $\tilde{d}$ is lexicographically greater than any symbol in $d$. Then $\tilde{d}$ satisfies the hypotheses of Lemma \ref{twoprops} and so there exists $\beta$ such that $\tilde{d} = d_\beta^*(1)$, and by Theorem \ref{PSSBeta}, $\{|P_{X_\beta}(n)|\} = \{\Phi_{\tilde{d}}(n)\}$. For each length $n$, $\tilde{d}$ will have one more n-letter word than the number occuring in $d$. Specifically, the first $n$ letters of $\tilde{d}$ cannot occur in $d$ because the first symbol of $\tilde{d}$ is not in the alphabet of $d$. Moreover, any word occuring in $\tilde{d}$ which does not see the first letter of $\tilde{d}$ is necessarily a subword of $d$, as $\sigma(\tilde{d}) = d$. Therefore, for any $n \in \mathbb{N}$, $\{\Phi_{\tilde{d}}(n)\} = \{\Phi_d(n) + 1\}$.
\end{proof}

Finally, we explore the extender set sequences of non-sofic $\beta$-shifts.
\begin{lemma}\label{wordsind}
Let $X_\beta$ be non-sofic. Then for any distinct words $w, v \in L_n(X_\beta)$ with $w \in L_n(d_\beta^*(1))$--that is, $w$ is a word actually appearing in the sequence $d_\beta^*(1)$--$w$ and $v$ have distinct extender sets.
\end{lemma}

\begin{proof}
If both words $w$ and $v$ appear in $d_\beta^*(1)$, $w$ and $v$ will have different predecessor sets--there will exist some n-letter word in $d_\beta^*(1)$ which one of the words is lexicographically less than or equal to but the other is not, namely, whichever of $w$ and $v$ is lexicographically lesser. Thus in this case, $w$ and $v$ have different extender sets. On the other hand, suppose $w \in L_n(d_\beta^*(1))$ and $v \notin L_n(d_\beta^*(1))$. If $v$ and $w$ have different follower or predecessor sets, they trivially have different extender sets, so suppose $P(w) = P(v)$ and $F(w) = F(v)$. (This implies that $w$ and $v$ end with the same maximal prefix of $d_\beta^*(1)$, and, by the above argument, that $v \prec w$. Note that obviously any digits on which $v$ and $w$ differ must occur before their shared suffix). Let $u$ be the shortest finite word such that $uw$ is a prefix of $d_\beta^*(1)$. \newline
\indent Choose a word $z \in F(w)$ which is lexicographically strictly larger than $(\sigma^{|uw|}(d_\beta^*(1)))_{|z|}$. Such a $z$ must exist: If there does not exist a legal word $z \in F(w)$ which is lexicographically strictly larger than $(\sigma^{|uw|}(d_\beta^*(1)))_{|z|}$, then either $w$ is the first $n$ letters of $d_\beta^*(1)$, in which case the requirement that $v$ and $w$ end in the same maximal prefix of $d_\beta^*(1)$ forces $v = w$, or else $(\sigma^{|uw|}(d_\beta^*(1)))_{|z|}$ is the lexicographically largest legal $|z|$-letter word in $X_\beta$, meaning that $(\sigma^{|uw|}(d_\beta^*(1)))_{|z|} = (d_\beta^*(1))_{|z|}$. But since this must be true for every possible value of $|z|$, we have that $\sigma^{|uw|}(d_\beta^*(1)) = d_\beta^*(1)$, and so $X_\beta$ is sofic. \newline
\indent After choosing such a $z$, it is clear that $(u,z)$ is not in the extender set of $w$. However, $(u,z)$ is in the extender set of $v$: $v$ is a synchronizing word, and since $P(w) = P(v)$ and $F(w) = F(v)$, $uv$ and $vz$ are legal, so $uvz$ is legal. 
\end{proof}

Note that, in the sofic case, if only one of $w$ and $v$ is in $L_n(d_\beta^*(1))$, $w$ and $v$ \textit{may} still have the same extender set:

\begin{example}\label{GMS2}
Recall that for the Golden Mean shift of Example \ref{GMS}, the one-sided shift is realizable as a $\beta$-shift for $\beta = \varphi = \frac{1+ \sqrt{5}}{2}$. In such a case, we have $d_\varphi^*(1) = .1010101010...$, a periodic sequence. It is clear that for a sequence $z$ on $\{0,1\}^{\mathbb{N}}$, having $\sigma^i(z) \preceq d_\varphi^*(1) \> \forall \> i \in \mathbb{N}$ is equivalent to $z$ having no two adjacent 1's. Moreover, the word $000$ clearly has the same extender set as the word $010$, despite $010$ appearing in $d_\varphi^*(1)$, showing that the result of Lemma \ref{wordsind} does not hold in the sofic case.
\end{example}

\begin{theorem}\label{nonsoficformula}
Let $X_\beta$ be a non-sofic $\beta$-shift. For $n \in \mathbb{N}$, let the classes $S_k, 0\leq k \leq n$, be as defined in the proof of Theorem \ref{UBFSSbeta}. Let $\eta(w,k): L_n(X_\beta) \times \{0, 1, ...n\} \rightarrow \{0,1\}$ be defined by 
$$\eta (w,k) = \begin{cases}
1 & \text{if } \exists \> v \in S_k, v\neq w \text{ s.t. } P_{X_\beta}(v) = P_{X_\beta}(w)\\
0 & \text{otherwise}.
\end{cases}
$$
Then for any $n \in \mathbb{N}$, $|E_{X_\beta}(n)| = \displaystyle \Phi_n(d_\beta^*(1)) + \sum_{w \in L_n(d_\beta^*(1))} \Big( \sum_{k=0}^n \eta (w,k) \Big)$.
\end{theorem}

\begin{proof}
Let $X_\beta$ be a non-sofic $\beta$-shift and $n \in \mathbb{N}$. By Lemma \ref{wordsind}, for any word $w$ in $L_n(d_\beta^*(1))$, and any word $v \in L_n(X_\beta), v \neq w$ implies $E_{X_\beta}(w) \neq E_{X_\beta}(v)$. Thus $X_\beta$ has $\Phi_n(d_\beta^*(1))$ distinct extender sets corresponding to words in $L_n(d_\beta^*(1)$, and moreover, any words not in $L_n(d_\beta^*(1))$ will have extender sets distinct from those of words in $L_n(d_\beta^*(1))$. Also, for any $\beta$-shift $X_\beta$ (sofic or not), if $w$ and $v$ both fail to be in $L_n(d_\beta^*(1))$, then $w$ and $v$ are both synchronizing, and so $P(w) = P(v)$ and $F(w) = F(v)$ is sufficient to prove $E(w) = E(v)$. Thus, the number of additional extender sets in $X_\beta$ corresponding to words not in $L_n(d_\beta^*(1))$ is the number of pairings of predecessor and follower sets achievable by those words. By Thm. ~\ref{PSSBeta}, each predecessor set is represented by some $w \in L_n(d_\beta^*(1))$. By Thm. ~\ref{UBFSSbeta}, the follower set of a word is determined by the class $S_k, \ 0 \leq k \leq n$, in which it lives. Hence the total number of possible pairings of predecessor and follower sets achievable by words not in $L_n(d_\beta^*(1))$ may be expressed by $\displaystyle \sum_{w \in L_n(d_\beta^*(1))} \Big( \sum_{k=0}^n \eta (w,k) \Big)$.
\end{proof}

\begin{remark}\label{soficincl}
The formula given in Theorem \ref{nonsoficformula} is an upper bound on the extender set sequence for sofic $\beta$-shifts. The only difference between the two is that in the sofic case, some of the extender sets of words in $L_n(d_\beta^*(1))$ may be the same as extender sets of words not in $L_n(d_\beta^*(1))$, so the formula may have overcounted.
\end{remark}



Though the formula in Theorem \ref{nonsoficformula} is exact, the following bounds on the extender set sequence of a non-sofic $\beta$-shift are simpler to express:

\begin{corollary}\label{ESBounds}
For any non-sofic $\beta$-shift $X_\beta$, $n \in \mathbb{N}$, $\Phi_n(d_\beta^*(1)) \leq |E_{X_\beta}(n)| \leq (n+1)\Phi_n(d_\beta^*(1))$.
\end{corollary}

\begin{proof}
Clearly for any $n \in \mathbb{N}$, $|E_{X_\beta}(n)| \geq |P_{X_\beta}(n)| = \Phi_n(d_\beta^*(1))$, proving the first inequality. For the second inequality, we use the equation from Corollary ~\ref{nonsoficformula}: 
$$|E_{X_\beta}(n)| = \displaystyle \Phi_n(d_\beta^*(1)) + \sum_{w \in L_n(d_\beta^*(1))} \Big( \sum_{k=0}^n \eta (w,k) \Big)$$
Recall that the class $S_n$ contains only one word, $(d_\beta^*(1))_n$. Therefore, for any $w \in L_n(d_\beta^*(1))$, $\eta(w,n) = 0$. Hence, \newline
$\begin{aligned}
|E_{X_\beta}(n)| &= \displaystyle \Phi_n(d_\beta^*(1)) + \sum_{w \in L_n(d_\beta^*(1))} \Big( \sum_{k=0}^{n-1} \eta (w,k) \Big)\\
&\leq \displaystyle \Phi_n(d_\beta^*(1)) + \sum_{w \in L_n(d_\beta^*(1))} n\\
&= \displaystyle \Phi_n(d_\beta^*(1)) + (\Phi_n(d_\beta^*(1))\cdot n\\
&= (n+1)\Phi_n(d_\beta^*(1)).
\end{aligned}$ \newline

\end{proof}

\begin{remark}
The above Corollary shows that for a non-sofic $\beta$-shift $X_\beta$, $|E_{X_\beta}(n)| \leq |F_{X_\beta}(n)| \cdot |P_{X_\beta}(n)|$ for any $n \in \mathbb{N}$.
\end{remark}




Two-sided $\beta$-shifts may be constructed from one-sided $\beta$-shifts using the natural extension; while the arguments of this paper are made specifically for one-sided shifts $X_\beta$, the reader may check that the following proposition provides the tools to generalize any of those results to two-sided $\beta$-shifts as well:

\begin{proposition}\label{two-sided} 
For any $\beta > 1$, let $X_\beta$ be the one-sided $\beta$-shift and $\widehat{X}_\beta$ be the two-sided $\beta$-shift corresponding to $\beta$. Then $L(X_\beta) = L(\widehat{X}_\beta)$.
\end{proposition}

\begin{proof}
By definition of the natural extension, $L(\widehat{X}_\beta) \subseteq L(X_\beta)$. Let $w \in L(X_\beta)$. Then there exists a one-sided sequence $z$ such that $z \in X_\beta$ and $w$ is a subword of $z$. But then for any $k$, $0^kz \in X_\beta$ as well. (If every shift of $z$ is lexicographically less than or equal to $d_\beta^*(1)$, then surely the same is true of $0^kz$). By taking limits of such sequences, we see that the sequence $0^\infty z$ is an element of $\widehat{X}_\beta$, and so $w \in L(\widehat{X}_\beta)$, completing the proof.
\end{proof}

\section*{acknowledgements} 

The author thanks his advisor, Dr. Ronnie Pavlov, for many helpful conversations.

\bibliography{biblio}{}

\begin{thebibliography}{1}

\bibitem{Blanchard}
F.~Blanchard.
\newblock {$\beta$}-expansions and symbolic dynamics.
\newblock {\em Theoret. Comput. Sci.}, 65(2):131--141, 1989.

\bibitem{FOP}
T.~{French}, N.~{Ormes}, and R.~{Pavlov}.
\newblock {Subshifts with Slowly Growing Numbers of Follower Sets}.
\newblock {\em ArXiv e-prints}, September 2015.

\bibitem{French}
Thomas French.
\newblock Characterizing follower and extender set sequences.
\newblock {\em Dyn. Syst.}, 31(3):293--310, 2016.

\bibitem{KassMadden}
Steve Kass and Kathleen Madden.
\newblock A sufficient condition for non-soficness of higher-dimensional
  subshifts.
\newblock {\em Proc. Amer. Math. Soc.}, 141(11):3803--3816, 2013.

\bibitem{LindMarcus}
Douglas Lind and Brian Marcus.
\newblock {\em An introduction to symbolic dynamics and coding}.
\newblock Cambridge University Press, Cambridge, 1995.

\bibitem{OrmesPavlov}
Nic Ormes and Ronnie Pavlov.
\newblock Extender sets and multidimensional subshifts.
\newblock {\em Ergodic Theory and Dynamical Systems}, FirstView:1--16, 4 2015.

\bibitem{Parry}
W.~Parry.
\newblock On the {$\beta $}-expansions of real numbers.
\newblock {\em Acta Math. Acad. Sci. Hungar.}, 11:401--416, 1960.

\bibitem{Renyi}
A.~R{\'e}nyi.
\newblock Representations for real numbers and their ergodic properties.
\newblock {\em Acta Math. Acad. Sci. Hungar}, 8:477--493, 1957.

\end{thebibliography}
\bibliographystyle{plain}

\end{document}